\documentclass[psamsfonts]{amsart}

\usepackage{amsthm}
\usepackage{amsmath,latexsym}
\usepackage{amssymb}
\usepackage{color}
\usepackage{setspace}
\usepackage[english, activeacute]{babel}
\usepackage[latin1]{inputenc}
\usepackage{float}
\usepackage{hyperref}
\usepackage{graphicx}
\usepackage{pdflscape}

\usepackage[left=2.7cm, right=2.7cm, top=2.5cm, bottom=2.7cm]{geometry}

\newtheorem{theo}{Theorem}
\newtheorem{ex}[theo]{Example}

\newtheorem{lem}[theo]{Lemma}
\newtheorem{cor}[theo]{Corollary}
\newtheorem{prop}[theo]{Proposition}
\newtheorem{rem}[theo]{Remark}

\newcommand{\PF}{\mathrm{PF}}
\newcommand{\SG}{\mathrm{SG}}
\newcommand{\g}{\mathrm{g}}
\newcommand{\F}{\mathrm{F}}
\newcommand{\G}{\mathrm{G}}
\newcommand{\m}{\mathrm{m}}
\newcommand{\Ap}{\mathrm{Ap}}

\newcommand{\Z}{\mathbb{Z}}
\newcommand{\N}{\mathbb{N}}

\newcommand{\K}{\mathbb{K}}
\newcommand{\sem}{\mathcal{S}}
\def\Maximals{{\rm Maximals}}
\def\Minimals{{\rm Minimals}_{\subseteq}}
\newcommand{\dsum}{\displaystyle\sum}

\usepackage[all]{xy}

\usepackage[ruled,vlined, algo2e]{algorithm2e}

\title{$m$-irreducible numerical semigroups}
\author[V. Blanco and J.C. Rosales]{V. Blanco$^\dag$ and J.C. Rosales$^\ddag$\bigskip\\
\lowercase{\texttt{vblanco@ugr.es}\; ; \;\texttt{jrosales@ugr.es}}}
\address{Departamento de \'Algebra. Universidad de Granada}
\email{$^\dag$vblanco@ugr.es; $^\ddag$jrosales@ugr.es}



\thanks{$\dag$The first author has been supported by the Juan de la Cierva program (Ref. JCI-2009-03896) and project MTM2007-67433-C02-01\\$\ddag$The second author has been supported by the project MTM2007-62346}
\date{\today}
\begin{document}
\maketitle
\begin{abstract}
In this paper we introduce the notion of $m$-irreducibility that extends the standard concept of irreducibility of a numerical semigroup when the multiplicity is fixed. We  analyze the structure of
the set of $m$-irreducible numerical semigroups, we give some properties of these numerical semigroups and we present algorithms to compute the decomposition of a numerical semigroups with multiplicity $m$ into $m$-irreducible numerical semigroups.

\end{abstract}

\section{Introduction}

A numerical semigroup is a subset $S$ of $\N$ (here $\N$ denotes the set of nonnegative integers) closed under
addition, containing zero and such that $\N\backslash S$ is finite.

A numerical semigroup is \textit{irreducible} if it cannot be expressed as an intersection of two numerical semigroups containing it properly. This notion was introduced in \cite{pacific} where
it is also shown that the family of irreducible numerical semigroups is the union of two families of numerical semigroups with special importance in this theory: symmetric and pseudo-symmetric numerical semigroups.

Every numerical semigroup can be expressed as a finite intersection of irreducible numerical semigroups. In \cite{forum}, the authors give an algorithm that allows to compute, for a numerical semigroup $S$, a finite set of irreducible numerical semigroups, $S_1, \ldots, S_n$, such that $S=S_1 \cap \cdots \cap S_n$. Furthermore, the above decomposition, given
by that algorithm, is minimal in the sense that $n$ is the smallest possible positive integer with that property. In \cite{belga}, upper and lower bounds are given for that $n$ in terms of $S$.

If $S$ is a numerical semigroup, the least positive integer belonging to $S$ is called the \textit{multiplicity} of $S$ and we denote it by $\m(S)$.

Let $m$ be a positive integer and $\mathcal{S}(m)$ the set of numerical semigroups with multiplicity $m$. It is clear that $(\sem(m), \cap)$ is a semilattice, that is, it is a semigroup where
all its elements are idempotent. The search of the generators of this semilattice leads us to the following definition. An element $S$ in $\sem(m)$ is $m$-irreducible if it cannot be
expressed as the intersection of two elements in $\sem(m)$ containing it properly. In Section \ref{sec:2} we prove that every element in $\sem(m)$ can be expressed as a finite intersection of $m$-irreducible numerical semigroups. Furthermore, we prove that
 $\{S: S \text{ is a $m$-irreducible numerical semigroup}\}$ coincides with $\{\{x\in \N: x\ge m\} \cup \{0\}\} \cup \{\{x\in \N: x\ge m \text{ and } x\neq i\} \cup \{0\}: i=m+1, \ldots, 2m-1\} \cup \{S: S \text{ is an irreducible numerical semigroup and } \m(S)=m\}$.

If $S$ is a numerical semigroup, the largest integer not belonging to $S$ is called the \textit{Frobenius number} of $S$ and we denote it by $\F(S)$. We say that a positive integer, $x$, is a \textit{gap} of $S$ if $x\not\in S$. We denote by $\G(S)$ the set of all the gaps of $S$. The
cardinal of $\G(S)$ is called the \textit{genus} of $S$, and we denote it by $\g(S)$. We can find in the literature several characterization for irreducible numerical semigroups. In \cite{springer}, the authors state that
the irreducible numerical semigroups are those with minimum genus among all the numerical semigroups with the same Frobenius number. In Section \ref{sec:3} we prove that this characterization is also valid for $m$-irreducible numerical semigroups. In fact, we prove that
 an element in $\sem(m)$ is $m$-irreducible if and only if it has minimum genus in the set of all the elements in $\sem(m)$ with the same Frobenius number.

Given a numerical semigroup $S$ with multiplicity $m$, we denote by $\mathcal{O}_m(S) = \{S' \in \sem(m): S \subseteq S'\}$ the set of \textit{oversemigroups} with multiplicity $m$ of $S$. The aim of Section
\ref{sec:4} is to provide an algorithm to compute $\mathcal{O}_m(S)$ when $S$ is given. These results are used in Section \ref{sec:5} to give an algorithm to compute a decomposition, as an intersection of $m$-irreducible numerical semigroups, of a numerical semigroup with
 multiplicity $m$. Moreover, the decomposition given by the algorithm is minimal in the sense that it uses the smallest number of $m$-irreducible numerical semigroups for taking part
of the decomposition.

For concluding this introduction, observe that although this work is performed by a ``semigroupist'' point of view, it can be used in Commutative Ring Theory. In fact, let $M$ be a submonoid of $(\N, +)$, $\K$ a field, and
$\K[[t]]$ the ring of formal power series over $\K$. It is well-known (see for instance, \cite{barucci}) that $\K[[M]]=\{\dsum_{s\in M} a_s t^s: a_s \in \K\}$ is a subring of $\K[[t]]$, called the ring
of the semigroup associated to $M$. All the above invariants, as the multiplicity, the genus (degree of singularity) and the Frobenius number (conductor minus one) have their corresponding
interpretation in this context (see \cite{barucci}). Moreover, in \cite{kunz} it is shown that a numerical semigroup is symmetric if and only if $\K[[S]]$ is a Gorenstein ring and in \cite{barucci-froberg} the rings $\K[[S]]$ when $S$ is pseudo-symmetric are called Kunz rings. Then, as a consequence
of the results in this paper we have that, given a numerical semigroup $S$ we can decompose the ring $\K[[S]]$ as intersection of rings with the same multiplicity where some of them are Gorenstein, some other are Kunz and others are rings associated to numerical
 semigroups with special simplicity: $\{x\in \N: x\ge m\} \cup \{0\}$ and $\{x\in \N: x\ge m \text{ and } x\neq i\} \cup \{0\}$ for $i \in \{m+1, \ldots, 2m-1\}$.

\section{$m$-irreducible numerical semigroups}
\label{sec:2}
We begin this section showing that every numerical semigroup with multiplicity $m$ can be expressed as a finite intersection of $m$-irreducible numerical semigroups.

Note that $\{x \in \N: x\geq m\} \cup \{0\}$ is the maximum
(with respect to the inclusion ordering) of $\sem(m)$ and then, this semigroup is $m$-irreducible.

\begin{prop}
\label{prop:1}
Let $S \in \sem(m)$. Then, there exist $S_1, \ldots, S_k$ $m$-irreducible numerical semigroups such that $S=S_1 \cap \cdots \cap S_k$.
\end{prop}
\begin{proof}
We prove the result by induction over $\g(S)$. If $\g(S)=m-1$, then $S=\{x \in \N: x\ge m\} \cup \{0\}$ and then  $m$-irreducible.
Assume that $\g(S)\ge m$ and that $S$ is no $m$-irreducible. Then there exist $S_1$ and $S_2$ in $\sem(m)$ such that $S \varsubsetneq S_1$, $S \varsubsetneq S_2$, and $S=S_1 \cap S_2$. By the induction hypothesis, there exist $S_{11}, \ldots, S_{1p}, S_{21}, \ldots, S_{2q}$ $m$-irreducible numerical semigroups such that $S_1 = S_{11} \cap \cdots \cap S_{1p}$ and $S_2 = S_{21} \cap \cdots \cap S_{2q}$.
 Then, $S=S_{11} \cap \cdots \cap S_{1p} \cap S_{21} \cap \cdots \cap S_{2q}$ is a decomposition of $S$ into $m$-irreducible numerical semigroups.
\end{proof}
Our next goal for this section is to prove Theorem \ref{theo:3} which states that the $m$-irreducible numerical semigroups are those maximal elements in $\sem(m)$ with the additional
condition that they have certain Frobenius number. In order to prove that result, we first introduce some notation and previous results.

Given two positive integer $m$ and $F$, we denote by $\sem(m,F)$ the set of numerical semigroups with multiplicity $m$ and Frobenius number $F$.
Note that $\sem(m,F) \neq \emptyset$ if and only if $F\ge m-1$ and $F$ is not a multiple of $m$.

Denote by $\sem^*(m,F)$ the set of maximal elements in $\sem(m,F)$  (with respect to the inclusion ordering).

The following result has an immediate proof and it is left to the reader

\begin{lem}
\label{lemma:2} $ $
\begin{enumerate}
 \item If $S\neq \N$ is a numerical semigroup, then $S \cup \{\F(S)\}$ is also a numerical semigroup.
\item Let $S_1, \ldots, S_n$ be numerical semigroups and $S=S_1\cap\cdots \cap S_n$. Then, $\F(S)=\max\{\F(S_1), \ldots, \F(S_n)\}$.
\end{enumerate}
\end{lem}

We are now ready to prove the announced result.

\begin{theo}
\label{theo:3}
Let $S \in \sem(m)$. Then, $S$ is $m$-irreducible if and only if $S \in \sem^*(m,\F(S))$.
\end{theo}
\begin{proof}
(Necessity) If $S=\{x \in \N: x\ge m\} \cup \{0\}$, clearly $S \in \sem^*(m,\F(S))$. Assume that $S\neq\{x \in \N: x\ge m\} \cup \{0\}$, then $\F(S)>m$. By Lemma \ref{lemma:2} we have that $S \cup \{\F(S)\} \in \sem(m)$. If $S\not\in \sem^*(m,\F(S))$ there exists $T\in \sem(m,\F(S))$ such that $S \varsubsetneq T$. It is clear that $S = T \cap (S \cup \{\F(S)\})$ contradicting that $S$ is $m$-irreducible.

(Sufficiency) Let $S \in \sem^*(m, \F(S))$. If $S$ is not $m$-irreducible then, there exist $S_1$ and $S_2$ in $\sem(m)$ such that $S\varsubsetneq S_1$, $S \varsubsetneq S_2$ and $S=S_1 \cap S_2$. By Lemma \ref{lemma:2}, $\F(S) \in \{\F(S_1), \F(S_2)\}$. Assume without loss of generality that $\F(S)=\F(S_1)$. Then, $S_1 \in \sem(m, \F(S))$ and $S \varsubsetneq S_1$ contradicting the maximality of $S$.
\end{proof}

In what follows we intend to give a proof for Proposition \ref{prop:6} where we describe explicitly the elements in $\sem^*(m, F)$. Before that, we give some previous results. The following lemma is well-known and appears in \cite{springer}.

\begin{lem}
\label{lemma:4}
Let $S$ be a numerical semigroup and assume that $h=\max\{ x \in \Z\backslash S: \F(S)-x \not\in S, x\neq \frac{F}{2}\}$ exists, then $S\cup \{h\}$ is a numerical semigroup with Frobenius number $\F(S)$.
\end{lem}

In \cite{pacific} it is shown that a numerical semigroup $S$ is irreducible if and only if it is maximal in the set of numerical semigroups with Frobenius number $\F(S)$. As a direct consequence of the above lemma we get the following result.
\begin{lem}
\label{lemma:5}
A numerical semigroup is irreducible if and only if $\{x \in \Z\backslash S: \F(S)-x \not\in S \text{ and } x \neq \dfrac{\F(S)}{2}\} = \emptyset$.
\end{lem}

We are ready to describe the structure of $\sem^*(m,F)$. Here, we denote by $a \equiv b \pmod c$ if $a-b$ is multiple of $c$.
\begin{prop}
\label{prop:6}
Let $m$ and $F$ be positive integers such that $F\ge m-1$ and $F \not\equiv 0 \pmod m$.
\begin{enumerate}
\item If $F=m-1$, then $\sem^*(m,F)=\big\{ \{x \in \N: x\ge m\} \cup \{0\} \big\}$.
\item If $m < F <2m$, then $\sem^*(m,F)=\big\{ \{x \in \N: x\ge m, x\neq F\} \cup \{0\} \big\}$.
\item If $F>2m$, then $\sem^*(m,F)=\big\{S \in \sem(m,F): S \text{ is irreducible }\big\}$.
\end{enumerate}
\end{prop}
\begin{proof}
$(1)$ and $(2)$ are trivial. Let us then prove $(3)$. It is clear, by the comment before Lemma \ref{lemma:5}, that $\{S \in \sem(F,m): S \text{ is irreducible }\big\} \subseteq \sem^*(m,F)$.
We prove the other inclusion. Let $S \in \sem^*(m,F)$. If $S$ is not irreducible, we know, by Lemma \ref{lemma:5}, that
$h=\max\{ x \not\in S: \F(S)-x \not\in S, x\neq \frac{F}{2}\}$ exists, and by Lemma \ref{lemma:4},
 that $S \cup \{h\}$ is numerical semigroup with Frobenius number $F$. Furthermore, $h> \frac{F}{2} > m$
 and then, $S \cup \{h\} \in \sem(m,F)$ contradicting the maximality of $S$.
\end{proof}

As a consequence of Theorem \ref{theo:3} and Proposition \ref{prop:6} we get this corollary.
\begin{cor}
\label{corollary:7}
A numerical semigroup, $S$, with multiplicity $m$ is $m$-irreducible if and only if one of the following conditions holds:
\begin{enumerate}
\item $S=\{x \in \N: x\ge m\} \cup \{0\}$.
\item $S=\{x \in \N: x\ge m, x\neq f\} \cup \{0\}$ with $f \in \{m+1, \ldots, 2m-1\}$.
\item $S$ is an irreducible numerical semigroup.
\end{enumerate}
\end{cor}

\section{The gaps of a $m$-irreducible numerical semigroup}
\label{sec:3}

Let $q$ be a rational number. We denote by $\lceil q \rceil = \min\{z \in \Z: q\leq z\}$ the ceiling part of $q$. It is well-known (see for instance \cite{springer}) that if $S$ is a numerical semigroup then  $\g(S) \ge \left\lceil \dfrac{\F(S)+1}{2} \right\rceil$. The following
result is easily deduced from Lemma \ref{lemma:5} and appears in \cite{springer}.
\begin{lem}
 \label{lemma:8}
 A numerical semigroup $S$ is irreducible if and only if $\g(S)= \left\lceil \dfrac{\F(S)+1}{2} \right\rceil$.
\end{lem}
As a consequence of Theorem \ref{theo:3}, Proposition \ref{prop:6} and Lemma \ref{lemma:8} we get the following result.

\begin{prop}
\label{prop:9}
If $S$ is a $m$-irreducible numerical semigroup, then
$$
\g(S) = \left\{\begin{array}{cl}
m-1 & \mbox{ if $\F(S)=m-1$,}\\
m  & \mbox{ if $m < \F(S) < 2m$,}\\
\left\lceil \dfrac{\F(S)+1}{2} \right\rceil & \mbox{ if $\F(S)>2m$}\\
\end{array}.\right.
$$
\end{prop}

By Lemma \ref{lemma:8}, the irreducible numerical semigroups are those with the smallest possible number of gaps once the Frobenius number is fixed.
 In the following result we prove that this property is extendable for $m$-irreducible numerical semigroups.

Let $m$ and $F$ be two positive integers such that $F \ge m-1$ and $F \not\equiv 0 \pmod m$. We denote by
$$
\g(m, F) = \min\{ \g(S): S \in \sem(m,F)\},
$$
the minimum genus among all the numerical semigroups with multiplicity $m$ and Frobenius number $F$.

\begin{theo}
\label{theo:10}
Let $S$ be a numerical semigroup with multiplicity $m$ and Frobenius number $F$. Then, $S$ is $m$-irreducible if and only if $\g(S)=\g(m,F)$.
\end{theo}
\begin{proof}
(Sufficiency) It is clear that if $\g(S)=\g(m, F)$ then $S$ is maximal in $\sem(m, F)$. Hence, $S \in \sem^*(m, F)$ and by Theorem \ref{theo:3} we have that $S$ is $m$-irreducible.

(Necessity) By Theorem \ref{theo:3} and Proposition \ref{prop:6} we distinguish the following three cases:
\begin{enumerate}
\item If $S=\{x \in \N: x\ge m\} \cup \{0\}$, it is clear that $\g(S)=m-1=\g(m, F)$.
\item If $S=\{x \in \N: x\ge m, x\neq F\} \cup \{0\}$ with $m<F<2m$, then, clearly  $\g(S)=m=\g(m, F)$.
\item If $S$ is irreducible and $F>2m$, then by Lemma \ref{lemma:8}, $\g(S)=\g(m, F)$.
\end{enumerate}
\end{proof}

From the above theorem we have that the $m$-irreducible numerical semigroups are exactly those with minimum genus among all the numerical semigroups with its same multiplicity and Frobenius number.

\begin{cor}
\label{corollary:11}
 Let $S$ be a numerical semigroup with multiplicity $m$. Then, $S$ is $m$-irreducible if and only if $\g(S) \in \big\{m-1, m, \left\lceil \dfrac{\F(S)+1}{2} \right\rceil\big\}$.
\end{cor}
\begin{proof}
 (Necessity) It is a direct consequence of Proposition \ref{prop:9}.

(Sufficiency) If $\g(S)=m-1$, then $S=\{x \in \N: x\ge m\} \cup \{0\}$ and by Corollary \ref{corollary:7} we have that $S$ is $m$-irreducible.

If $\g(S)=m$ then we deduce that $m<\F(S)<2m$, and $S=\{x \in \N: x\ge m, x\neq \F(S)\} \cup \{0\}$ and again by applying Corollary \ref{corollary:7}, $S$ is $m$-irreducible.

If $\g(S)=\left\lceil \dfrac{\F(S)+1}{2} \right\rceil$, then by
Corollary \ref{corollary:7} and Lemma \ref{lemma:8}, we conclude that $S$ is $m$-irreducible.
\end{proof}

Let $S$ be a numerical semigroup. We say that an element $x\in \N\backslash S$ is an special gap of $S$ if $S \cup \{x\}$ is a numerical semigroup. We denote by $\SG(S)$ the set of all the special
gaps of $S$. Let $A$ be a set, we denote by $|A|$ the cardinal of $A$.

The following result appears in \cite{springer}.

\begin{lem}
 \label{lemma:12}
Let $S$ be a numerical semigroup.
\begin{enumerate}
 \item $S$ is irreducible if and only if $|\SG(S)|\leq 1$.
\item If $T$ is a numerical semigroup such that $S \varsubsetneq T$ and $x=\max (T\backslash S)$, then $x \in \SG(S)$.
\end{enumerate}
\end{lem}

From the above lemma we have that the irreducible numerical semigroups are characterized by the cardinal of the set of its special gaps. Next, we show that the
$m$-irreducible numerical semigroups are characterized by the cardinal of its special gaps that are greater than $m$.

\begin{theo}
 \label{theo:13}
Let $S$ be a numerical semigroup with multiplicity $m$. Then, $S$ is $m$-irreducible if and only if $|\{x \in \SG(S): x>m\}| \leq 1$.
\end{theo}
\begin{proof}
 (Necessity) If $S$ is a $m$-irreducible numerical semigroup, then, by Theorem \ref{theo:3} and Proposition \ref{prop:6}, one of the following condition holds:
\begin{enumerate}
\item If $S=\{x \in \N: x\ge m\} \cup \{0\}$, it is clear that $|\{x \in \SG(S): x>m\}| =0 \leq 1$.
\item If $S=\{x \in \N: x\ge m, x\neq \F(S)\} \cup \{0\}$ with $m<\F(S)<2m$, then, clearly  $|\{x \in \SG(S): x>m\}| = |\{\F(S)\}| = 1$.
\item If $S$ is irreducible and $\F(S)>2m$, then by Lemma \ref{lemma:12}, $|\{x \in \SG(S): x>m\}| \leq 1$.
\end{enumerate}

(Sufficiency) If $S$ is not an $m$-irreducible numerical semigroup then there exist $S_1, S_2 \in \sem(m)$ such that $S \varsubsetneq S_1$, $S\varsubsetneq S_2$ and $S= S_1 \cap S_2$.
 For $i \in \{1, 2\}$, let $x_i = \max(S_i\backslash S)$. Because $S=S_1 \cap S_2$ we have that $x_1\neq x_2$ and by Lemma \ref{lemma:12}, we conclude that $|\{x \in \SG(S): x>m\}| \geq 2$ contradicting the hypothesis.
\end{proof}

In the rest of the section we denote by $T$ a numerical semigroup and $n \in T\backslash\{0\}$. The Apery set~\cite{apery} of $T$ with respect to $n$ is $\Ap(T,n)=\{t \in T: t-n \not\in T\}$. Our goal is to show how to determine $\SG(T)$ when $\Ap(T,n)$ is known,. This process will be useful for the rest of this paper.

It is well-known and easy to prove that $\Ap(T,n)=\{w(0), \ldots, w(n-1)\}$, where $w(i)$ is the least element in $T$ congruent with $i$ modulo $n$, If we denote by $a \mod b$ the remainder of division of $a$ by $b$,
 then a positive integer, $x$, is in $T$ if and only if $w(x \mod n)\leq x$ (see for instance \cite{springer}).

By using the notation introduced in \cite{jpaa}, an integer $x$ is a pseudo-Frobenius number of $T$ if $x \not\in T$ and $x+t \in T$ for all $t\in T\backslash \{0\}$. We denote by $\PF(T)$ the set of all pseudo-Frobenius number of $T$. The cardinality of the above set is an important invariant of $T$, called the \textit{type} of $T$ (see \cite{barucci}). The relationship between $PF(T)$ and $\SG(T)$ is shown in the following lemma, whose proof is immediate.

\begin{lem}
\label{lemma:14}
With the above conditions $\SG(T)=\{x \in \PF(T): 2x\in T\}$.
\end{lem}

Over $\N$, we can define the following order relation:
$$
a \leq_T b \text{ if } b-a\in T
$$
The following result appears in \cite{jpaa} and characterizes the set of pseudo-Frobenius numbers of a numerical semigroup in terms of one of its Ap\'ery sets.

\begin{lem}
\label{lemma:15}
 With the above notation $\PF(T)=\{w-n : w \in \Maximals_{\leq_T} \Ap(T, n)\}$.
\end{lem}
Note that $w(i)\in \Maximals_{\leq_T} \Ap(T, n)$ if and only if $w(k)-w(i) \not\in \Ap(T,n)$ for all $k\neq i$.

We finish this section by giving an algorithm that allows to determine the special gaps of a numerical semigroup when the Ap\'ery set with respect a non-zero element of the semigroup of is given. Algorithm \ref{alg:sg} shows a pseudocode for this computation.

\vspace*{0.3cm}
\begin{algorithm2e}[H]
\label{alg:sg}

\SetKwInOut{Input}{Input}
\SetKwInOut{Output}{Output}
\Input{A numerical semigroup $T$ and its Ap\'ery set with respect to $n\in T\backslash\{0\}$, $\Ap(T, n)=\{w(0), \ldots, w(n-1)\}$.}

Compute $M=\{w(i)-n: w(k)-w(i) \not\in \Ap(T,n), \text{ for all } k\neq i\}$.

\Output{$\SG(T)=\{m \in M: 2m \ge w(2m \mod n)\}$.}
\caption{Computing the special gaps when an Ap\'ery set is given.}
\end{algorithm2e}

Next, we illustrate the usage of the above algorithm in a simple example.
\begin{ex}
\label{example:17}
 Let us compute the special gaps of the numerical semigroup $S=\{0, 5,$ $7, 9,$ $10, 12,$ $14, \rightarrow\}$ by using Algorithm \ref{alg:sg}. It is clear that $\Ap(S, 5)=\{0,7,9,16,18\}$. Then, $M=\{16-5=11, 18-5=13\}$ and $\SG(S)=\{11,13\}$.
\end{ex}

\section{The oversemigroups of multiplicity $m$ of a numerical semigroup}
\label{sec:4}
Let $S$ be a numerical semigroups with multiplicity $m$. We denote by $\mathcal{O}_m(S)=\{S'\in \sem(m): S\subseteq S'\}$ the set of oversemigroups with multiplicity $m$ of $S$.

The elements in $\mathcal{O}_m(S)$ can be obtained recursively. The main idea for this construction is the following result that can be directly obtained from Lemma \ref{lemma:12}.

\begin{lem}
 \label{lemma:18}
Let $S$ and $S'$ be elements in $\sem(m)$ such that $S\varsubsetneq S'$ and let $x=\max (S'\backslash S)$. Then $S \cup \{x\} \in \sem(m)$.
\end{lem}

Given two numerical semigroups $S$ and $S'$ in $\sem(m)$ with $S\varsubsetneq S'$, we define recursively the following sequence of elements in $\sem(m)$:
\begin{itemize}
 \item $S_0=S$.
\item $S_{n+1} = \left\{ \begin{array}{cl}
                          S_n & \mbox{if $S_n = S'$},\\
			  S_n \cup \{\max (S'\backslash S_n)\} & \mbox{ otherwise.}
                         \end{array}\right.$
\end{itemize}
It is clear that if $|S'\backslash S|=k$, then $S=S_0 \varsubsetneq S_1 \varsubsetneq \cdots \varsubsetneq S_k = S'$.

With this idea it is not difficult to design an algorithm that allows to compute all the elements in $\mathcal{O}_m(S)$ for a given numerical semigroup with multiplicity $m$, $S$. The main idea for this procedure is that if $S'\in\mathcal{O}_m(S)$ (we start with $S'=S$) and $\{x_1, \ldots,, x_r\}=\{x \in \SG(S'): x>m\}$ then $S' \cup \{x_1\}, \ldots, S' \cup \{x_r\}$ are also in $\mathcal{O}_m(S)$. Algorithm \ref{alg:over} shows a pseudocode for this procedure.

\vspace*{0.3cm}

\begin{algorithm2e}[H]
\label{alg:over}
\SetKwInOut{Input}{Input}
\SetKwInOut{Output}{Output}
\SetKwInOut{Initialization}{Initialization}
\Input{A numerical semigroup $S$ with multiplicity $m$}

\Initialization{$A=\{S\}$ and $B=\{S\}$.}

\While{$B\neq \emptyset$}{
\begin{itemize}
 \item For each $S'\in B$ compute $D(S')=\{x \in \SG(S'): x>m\}$.
\item  Set $A := A \cup \{\bigcup_{S'\in B} \{S' \cup \{x\}: x \in D(S')\}\}$.
\item Set $B:= \bigcup_{S'\in B} \{S' \cup \{x\}: x \in D(S')\}$.
\end{itemize}}

\Output{$A=\mathcal{O}_m(S)$.}
\caption{Computing the oversemigroups with multiplicity $m$ of a numerical semigroup.}
\end{algorithm2e}

\vspace*{0.3cm}

Observe that the difficulty of the above algorithm  lies on the computation of $\{x \in \SG(S'): x>m\}$.
 Recall that if we know the Ap\'ery set $\Ap(S, m)$, by using Algorithm \ref{alg:sg} we can easily compute the set $\{x\in \SG(S'): x>m\}$.

The following result states that if we know $\Ap(S, m)$, then we can also compute $\Ap(S', m)$ for all those oversemigroups $S'$ that appears when we run Algorithm \ref{alg:over}.

\begin{lem}
 \label{lemma:20}
Let $T \in \sem(m)$ and $\Ap(T, m)=\{w(0), \ldots, w(m-1)\}$. If $x \in \SG(T)$, then
$$
\Ap(T \cup \{x\}, m) = (\Ap(T,m) \backslash \{ w(x \mod m)\}) \cup \{w(x \mod m) - m\}
$$
\end{lem}

\begin{proof}
 Note that by Lemma \ref{lemma:15} we know that $x=w(x \mod m) -m$.
\end{proof}

Observe that a numerical semigroup with multiplicity $m$, $S$, is completely determined by $\Ap(S,m)$. Hence, $S$ is also completely determined by a $(m-1)$-tuple $(w(1), \ldots, w(m-1))$ in $\N^{m-1}$ where $w(i)$ is the least element in $S$ congruent with $i$ modulo $m$. We will call in the rest of the paper to
 $(w(1), \ldots, w(m-1))$ the \textit{coordinates} of $S$.

We denote by $[m-1]=\{1, \ldots, m-1\}$, and for each $i \in [m-1]$ by $e_i$, the $(m-1)$-tuple having a $1$ as its $i$th entry and zeros otherwise. With this notation, the following result
is a reformulation of Lemma \ref{lemma:20}.

\begin{lem}
 \label{lemma:21}
Let $T \in \sem(m)$ and $x\in \SG(T)$ such that $x>m$. If $(x_1, \ldots, x_{m-1})$ are the coordinates of $T$, then $(x_1, \ldots, x_{m-1})-m\cdot e_{x\mod m}$ are the  coordinates of $T \cup \{x\}$.
\end{lem}

The reader can easily check that the following algorithm computes, from the coordinates of a numerical semigroup with multiplicity $m$, $S$, the set of all the coordinates of the elements in $\mathcal{O}_m(S)$.

\vspace*{0.3cm}

\begin{algorithm2e}[h]
\label{alg:coord}
\SetKwInOut{Input}{Input}
\SetKwInOut{Output}{Output}
\SetKwInOut{Initialization}{Initialization}
\Input{The coordinates $x=(x_1, \ldots, x_{m-1})$ of a numerical semigroup $S$ with multiplicity $m$}

\Initialization{$A=\{x\}$ and $B=\{x\}$.}

\While{$B\neq \emptyset$}{
\begin{itemize}
 \item For each $y=(y_1, \ldots, y_{m-1})\in B$ compute

 \noindent$D(y)=\{i \in [m-1]: y_i>2m \text{ and } y_k-y_i \not\in \{y_1, \ldots, y_{m-1}\} \text{  for all } k \in [m-1]\}$.
\item  Set $A := A \cup \{\bigcup_{y \in B} \{y - m\cdot e_i: i \in D(y)\}\}$.
\item Set $B:= \bigcup_{y \in B} \{y - m\cdot e_i: i \in D(y)\}$.
\end{itemize}}

\Output{$A=$ Coordinates of $\mathcal{O}_m(S)$.}
\caption{The coordinates of all the oversemigroups with multiplicity $m$ of a numerical semigroup.}
\end{algorithm2e}

We finish the section by illustrating the above algorithm with an example.

\begin{ex}
 \label{example:23}
Let $S=\{0,5,7,9,10,12,14, \rightarrow\}$. It is clear that $S$ is a numerical semigroup with multiplicity $5$ and that $(16, 7, 18,9)$ are its coordinates.
\begin{itemize}
 \item We start by initializing $A=\{(16, 7,18,9)\}$ and $B=\{(16, 7,18,9)\}$.

 \noindent$D(16,7,18,9) = \{1,3\}$.
\item $A=\{(16, 7,18,9), (11,7,18,9), (16, 7, 13, 9)\}$ and $B=\{(11,7,18,9), (16, 7,13,9)\}$.

 \noindent$D(11,7,18,9) = \{3\}$ and $D(16, 7, 13, 9)=\{1,3\}$.
\item $A=\{(16, 7,18,9), (11,7,18,9), (16, 7, 13, 9), (11,7,13,9), (16,7,8,9)\}$ and \\
$B=$ $\{(11,7,13,9),$ $(16, 7,8,9)\}$.

 \noindent$D(11,7,13,9) = \{1,3\}$ and $D(16, 7, 8, 9)=\{1\}$.
\item $A=\{(16, 7,18,9), (11,7,18,9), (16, 7, 13, 9), (11,7,13,9)$, $(16,7,8,9),$ $(6, 7, 13, 9),$ $(11, 7, 8, 9)\}$ and $B=\{(6,7,13,9), (11, 7,8,9)\}$.

 \noindent$D(6,7,13,9) = \{3\}$ and $D(11, 7, 8, 9)=\{1\}$.
\item $A=\{(16, 7,18,9), (11,7,18,9), (16, 7, 13, 9), (11,7,13,9),$ $(16,7,8,9),$ $(6, 7, 13, 9),$ $(11, 7, 8, 9),$ $(6,7,8,9)\}$ and $B=\{(6,7,8,9)\}$.

 \noindent$D(6,7,8,9) = \emptyset$.
\item $A=\{(16, 7,18,9), (11,7,18,9), (16, 7, 13, 9), (11,7,13,9),$ $(16,7,8,9),$ $(6, 7, 13, 9),$ $(11, 7, 8, 9),$ $(6,7,8,9)\}$ and $B= \emptyset$.
\end{itemize}
And then, the coordinates of all the oversemigroups with multiplicity $5$ of $S$ are
$\{(16, 7,18,9),$ $(11,7,18,9),$ $(16, 7, 13, 9),$ $(11,7,13,9),$ $(16,7,8,9),$ $(6, 7, 13, 9),$ $(11, 7, 8, 9),$ $(6,7,8,9)\}$
\end{ex}

\section{Decomposition into $m$-irreducible numerical semigroups}
\label{sec:5}
Let $S$ be a numerical semigroup with multiplicity $m$. We denote by $\mathcal{J}_m(S) = \{S'\in \mathcal{O}_m(S): S' \text{ is $m$-irreducible}\}$. From Proposition \ref{prop:1} we deduce that $S=\bigcap_{S'\in \mathcal{J}_m(S)} S'$. The following result has an immediate proof.

\begin{lem}
 \label{lemma:24}
Let $S \in \sem(m)$ and $\{S_1, \ldots, S_n\}$ the set of all the minimal elements (with respect to the inclusion ordering) of $\mathcal{J}_m(S)$. Then $S=S_1 \cap \cdots \cap S_n$.
\end{lem}

As a consequence of the results in Section \ref{sec:3} and Theorem \ref{theo:13} we have justified the following algorithm that computes the minimal elements in $\mathcal{J}_m(S)$.

\vspace*{0.3cm}

\begin{algorithm2e}[H]
\label{alg:min}
\SetKwInOut{Input}{Input}
\SetKwInOut{Output}{Output}
\SetKwInOut{Initialization}{Initialization}
\Input{A numerical semigroup $S$ with multiplicity $m$}

\Initialization{$A=\{S\}$ and $B=\emptyset$.}

\While{$A\neq \emptyset$}{
For each $S' \in A$ compute $D(S')=\{x \in \SG(S'): x>m\}$.

\eIf{$|D(S')|\leq 1$}{Set $B:=B \cup \{S'\}$ and $A:=A\backslash \{S'\}$.}{Set $A:=(A\backslash\{S'\}) \cup \{S'\cup \{x\}: x \in D(S') \text{ and } S'\cup\{x\} \text{ does not contain any element of $B$}\}$.}}

\Output{$B= \Minimals(\mathcal{J}_m(S))$}
\caption{Computation of $\Minimals(\mathcal{J}_m(S))$.}
\end{algorithm2e}

\vspace*{0.3cm}

We illustrate the above algorithm with the following example. Note that if $S$ and $T$ are elements in $\sem(m)$ with coordinates $x$ and $y$, respectively, then $S\subseteq T$ if and only if $y \leq x$.

\begin{ex}
 \label{example:26}
Let $S=\{0, 5, 7, 9, 10, 12, 14, \rightarrow\}$ the numerical semigroup given in Example \ref{example:23}. We compute the set $\Minimals(\mathcal{J}_m(S))$.  Recall that $(16, 7, 18, 9)$ are the coordinates of $S$. Then, by running Algorithm \ref{alg:min} we obtain:
\begin{itemize}
 \item $A=\{(16, 7,18,9)\}$ and $B=\emptyset$. $D(16,7,18,9)=\{11,13\}$.
\item  $A=\{(11, 7,18,9), (16, 7, 13, 9)\}$ and $B=\emptyset$. $D(11,7,18,9)=\{13\}$.
\item  $A=\{(16, 7, 13, 9)\}$ and $B=\{(11, 7, 18, 9)\}$. $D(16,7,13,9)=\{11, 8\}$.
\item  $A=\{(16, 7, 8, 9)\}$ and $B=\{(11, 7, 18, 9)\}$. $D(16,7,8,9)=\{11\}$.
\item $A=\emptyset$ and $B=\{(11,7,18, 9), (16, 7, 8, 9)\}$.
\end{itemize}
\end{ex}

Note that the decomposition described in Lemma \ref{lemma:24} is not necessarily minimal, in the sense that the smallest number of $m$-irreducible numerical semigroups are taking part of the decomposition. An example of this fact is shown in Example \ref{example:31}.

To compute the minimal decomposition are necessary the two following results.

\begin{lem}
 \label{lemma:27}
Let $S \in \sem(m)$. If $S=S_1\cap \cdots S_n$ with $S_1, \ldots, S_m \in \mathcal{J}_m(S)$, then there exist $S_1', \ldots, S_n' \in \Minimals(\mathcal{J}_m(S))$ such that $S=S_1'\cap \cdots \cap S_n'$.
\end{lem}
\begin{proof}
 Take, for each $i \in \{1, \ldots, n\}$ $S'_i$ the element in $\Minimals(\mathcal{J}_m(S))$ such that $S_i'\subseteq S_i$.
\end{proof}

\begin{lem}
 \label{lemma:28}
Let $S\in \sem(m)$ and $S_1, \ldots, S_n \in \mathcal{O}_m(S)$. Then, $S=S_1 \cap \cdots \cap S_n$ if and only if for all $h \in \{x \in \SG(S): x>m\}$ there exists $i \in \{1, \ldots, n\}$ such that $h \not\in S_i$.
\end{lem}
\begin{proof}
 It is a direct consequence of Lemma \ref{lemma:12}.
\end{proof}

\begin{rem}
 \label{remark:29}
Note that as a direct consequence of the above lemma we have that  if $S \in \sem(m)$, then $S$ can be expressed as an intersection less or equal than $|\{x \in \SG(S): x>m\}|$ $m$-irreducible numerical semigroups. In fact, by Lemma \ref{lemma:14} and Corollary 1.23 in \cite{springer} we get that we can decompose $S$ into less or equal that $m-1$ $m$-irreducible numerical semigroups.
\end{rem}

Assume that $\Minimals(\mathcal{J}_m(S)) = \{S_1, \ldots, S_n\}$. For each $i \in \{1, \ldots, n\}$, let $P(S_i)=\{h \in \SG(S): h>m \text{ and } h\not\in S_i\}$. By Lemma \ref{lemma:28}, we know that $S=S_{i_{1}} \cap \cdots \cap S_{i_{r}}$ if and only if $P(S_{i_{1}}) \cup \cdots \cup P(S_{i_{r}}) = \{x \in \SG(S): x>m\}$. This comment and Lemma \ref{lemma:27} justify the following algorithm that allows to
 compute from a given numerical semigroup with multiplicity $m$, a minimal decomposition as intersection of $m$-irreducible numerical semigroups.

\vspace*{0.3cm}

\begin{algorithm2e}[H]
\label{alg:final}
\SetKwInOut{Input}{Input}
\SetKwInOut{Output}{Output}

\Input{A numerical semigroup $S$ with multiplicity $m$}

\begin{enumerate}
 \item Compute $\Minimals(\mathcal{J}_m(S))$ by Algorithm \ref{alg:min}.
\item For each $S'\in \Minimals(\mathcal{J}_m(S))$, compute $P(S')=\{h \in \SG(S): h>m \text{ and } h\not\in S'\}$.
\item Choose $A \subseteq \Minimals(\mathcal{J}_m(S))$ with minimal cardinality and such that $\bigcup_{S'\in A} P(S') = \{x \in \SG(S): x>m\}$.
\end{enumerate}

\Output{$A=\{S_1, \ldots, S_n\}$ such that $S=S_1 \cap \cdots \cap S_n$ is a minimal decomposition of $S$ in $m$-irreducible numerical semigroups.}
\caption{Computation of a minimal decomposition of a numerical semigroup with multiplicity $m$ in $m$-irreducible numerical semigroups.}
\end{algorithm2e}

\vspace*{0.3cm}

The following two examples shows the usage of the above methodology.

\begin{ex}
\label{example:31}
 Let $S = \{0, 5, 10, 11, 14, 15, 16, 19, 20, 21, 22, 24, \rightarrow\}$. The coordinates for $S$ are $x=(11, 22, 28, 14)$. By using Algorithm \ref{alg:min}, we obtain that the coordinates of the
elements in $\Minimals(\mathcal{J}_m(S))$ are
$$
(11, 22, 8, 14), (11, 22, 13, 9), (11, 17, 28, 14)
$$
and then, we have a decomposition into the three above $5$-irreducible numerical semigroups of $S$. However, if we apply Algorithm \ref{alg:final}:
\begin{itemize}
 \item $P(11, 22, 8, 14)=\{17\}$,
\item $P(11, 22, 13, 9)=\{17\}$,
\item $P(11, 17, 28, 14)=\{23\}$,
\end{itemize}
while $\{ x \in \SG(S): x>5\} =\{17, 23\}$, so $(11, 22, 13, 9)$ and  $(11, 17, 28, 14)$ are enough to decompose $S$ into $5$-irreducible numerical semigroups.
\end{ex}

In the above example, both $5$-irreducible numerical semigroups in the decomposition are also irreducible. In the next example, we show that it is not true in general.

\begin{ex}
\label{example:31}
 Let $S= \{0, 9, 17, 18, 24, \rightarrow\}$. The coordinates for $S$ are $x=(28, 29, 30, 31, 32, 24, 25, 17)$. By using Algorithm \ref{alg:min}, we obtain that the coordinates of the
elements in $\Minimals(\mathcal{J}_m(S))$ are:
\begin{center}
$( 10, 20, 21, 31, 14, 15, 16, 17),$ $( 10, 20, 12, 22, 32, 15, 16, 17),$ $( 19, 11, 21, 13, 32, 15, 16, 17),$ $( 19, 11, 30, 13, 14, 15, 16, 17),$
 $( 19, 20, 12, 13, 32, 15, 16, 17),$ $( 28, 11, 12, 13, 14, 15, 16, 17),$ $( 10, 20, 30, 13, 14, 15, 16, 17),$ $( 10, 11, 12, 13, 14, 24, 16, 17),$
 $( 19, 29, 12, 13, 14, 15, 16, 17),$ $( 10, 11, 12, 13, 14, 15, 25, 17),$ $( 10, 11, 21, 22, 32, 15, 16, 17)$ and $( 19, 20, 12, 31, 14, 15, 16, 17)$,
\end{center}
and then, we have a decomposition into the above $9$-irreducible numerical semigroups of $S$. However, if we apply Algorithm \ref{alg:final} we obtain these $12$ minimal $9$-irreducible oversemigroups of $S$:
\begin{itemize}
\item $P( 10, 20, 21, 31, 14, 15, 16, 17)=\{22\}$,
\item $P( 10, 20, 12, 22, 32, 15, 16, 17)=\{23\}$,
\item $P( 19, 11, 21, 13, 32, 15, 16, 17)=\{23\}$,
\item $P( 19, 11, 30, 13, 14, 15, 16, 17)=\{21\}$,
\item $P( 19, 20, 12, 13, 32, 15, 16, 17)=\{23\}$,
\item $P( 28, 11, 12, 13, 14, 15, 16, 17)=\{19\}$,
\item $P( 10, 20, 30, 13, 14, 15, 16, 17)=\{21\}$,
\item $P( 10, 11, 12, 13, 14, 24, 16, 17)=\{15\}$,
\item $P( 19, 29, 12, 13, 14, 15, 16, 17)=\{20\}$,
\item $P( 10, 11, 12, 13, 14, 15, 25, 17)=\{16\}$,
\item $P( 10, 11, 21, 22, 32, 15, 16, 17)=\{23\}$,
\item $P( 19, 20, 12, 31, 14, 15, 16, 17)=\{22\}$.
\end{itemize}
Because $\{ x \in \SG(S): x>9\} =\{15, 16, 19, 20, 21, 22, 23\}$, then $( 10, 11, 12, 13, 14, 24, 16, 17)$,\\
 $( 10, 11, 12, 13, 14, 15, 25, 17)$, $( 28, 11, 12, 13, 14, 15, 16, 17)$, $( 19, 29, 12, 13, 14, 15, 16, 17)$,\\
  $( 19, 11, 30, 13, 14, 15, 16, 17)$, $( 10, 20, 21, 31, 14, 15, 16, 17)$  and  $( 19, 11, 21, 13, 32, 15, 16, 17)$ are enough to decompose $S$ into $9$-irreducible numerical semigroups.

Note that the standard decomposition into irreducible numerical semigroups is given by the coordinates $(9, 18, 19, 12, 13, 22, 31),$ $( 10, 11, 12, 13, 14, 15, 25, 17 ),$ $( 28, 11, 12, 13, 14, 15, 16, 17 ),$ $( 19, 11, 30, 13, 14, 15, 16, 17 ),$
  $( 19, 29, 12, 13, 14, 15, 16, 17 )$ and $( 19, 20, 12, 31, 14, 15, 16, 17 )$, where $(9, 18, 19, 12, 13, 22, 31)$ has multiplicity $8$, and then, it is not a decomposition into $9$-irreducibles.

\end{ex}

\begin{rem}
 Note that analogously to the extension done in this paper for irreducible numerical semigroups when the multiplicity is fixed, we could also extend the notions of symmetry and pseudosymmetry as follows:
\begin{itemize}
 \item $S \in \sem(m)$ is $m$\textit{-symmetric} if $S$ is $m$-irreducible and $\F(S)$ is odd.
 \item $S \in \sem(m)$ is $m$\textit{-pseudosymmetric} if $S$ is $m$-irreducible and $\F(S)$ is even.
\end{itemize}
Moreover, by using Proposition \ref{prop:6}, we can describe, for a given multiplicity $m$, all the $m$-symmetric and $m$-pseudosymmetric numerical semigroups
in terms of the Frobenius number. Denote by ${\rm Symm}(m)$ and ${\rm PSymm}(m)$ the set of $m$-symmetric and $m$-pseudosymmetric numerical semigroups, respectively.
Let $S \in \sem(m)$ and $\F(S)=F$.
\begin{enumerate}
\item If $F=m-1$, then $S \in {\rm Symm}(m)$ (resp. $S \in {\rm PSymm}(m)$) if and only if $S=\{x \in \N: x\ge m\} \cup \{0\}$ and $m$ is even (resp. $m$ is odd).
\item If $m < F <2m$, then $S \in {\rm Symm}(m)$ (resp. $S \in {\rm PSymm}(m)$) if and only if $S=\{x \in \N: x\ge m, x\neq F\} \cup \{0\}$ and $F$ is odd (resp. $F$ is even).
\item If $F>2m$, then $S \in {\rm Symm}(m)$ (resp. $S \in {\rm PSymm}(m)$) if and only if $S$ is symmetric (resp. $S$ is pseudosymmetric).
\end{enumerate}
\end{rem}

\end{document}